%% file: seifert_invts.tex
\documentclass[11pt]{amsart}

\usepackage{amsmath}
\usepackage[font=small]{caption}
\usepackage{amsthm}
\usepackage{amsfonts}
\usepackage[pdftex]{graphicx}
\usepackage[margin=.5cm, format=hang]{subfig}
\usepackage[alphabetic]{amsrefs}
\usepackage{amssymb}                                                                                                                                                                                                                                                                                                                                                                                                                                                                                                                                                                                                                                                                                                                                                                                                 
\usepackage{pinlabel}
\usepackage{morefloats}
\usepackage{enumerate}
\usepackage{fullpage}
\usepackage{array}
\usepackage{multirow}
\usepackage{sidecap}
\usepackage{hhline}
\usepackage{multirow}
\usepackage{blkarray}
\usepackage{xcolor}
\usepackage{comment}
\usepackage{pinlabel}
\usepackage{subfig}
\usepackage{nicefrac}

\newtheorem{thm}{Theorem}[section]
\newtheorem{rmk}[thm]{Remark}
\newtheorem{prop}[thm]{Proposition}
\newtheorem{lem}[thm]{Lemma}
\newtheorem{ex}[thm]{Example}

\newtheorem{cor}[thm]{Corollary}

\usepackage{mathtools}

\DeclarePairedDelimiter{\abs}{\lvert}{\rvert}
\DeclarePairedDelimiter{\pair}{\langle}{\rangle}

\newcommand{\Z}{\mathbb{Z}}

\title{A note on concordance properties of fibers in Seifert homology spheres}

\author[Tye Lidman]{Tye Lidman}
\address{Department of Mathematics, North Carolina State University, Raleigh, NC 27607}
\email{tlid@math.ncsu.edu}

\author[Eamonn Tweedy]{Eamonn Tweedy}
\address{Department of Mathematics, Widener University, Chester, PA 19013}
\email{etweedy@widener.edu}

\begin{document}
\maketitle

\section*{Abstract} In this note, we collect various properties of Seifert homology spheres from the viewpoint of Dehn surgery along a Seifert fiber.  We expect that many of these are known to various experts, but include them in one place which we hope to be useful in the study of concordance and homology cobordism.

\input{intro}

\input{background}

\input{surgery-new}

\input{dinvariants}

\input{topslice}

\bibliography{references}

\end{document}

%% file: intro.tex
\section{Introduction}

Seifert fibered spaces play an important role in three-manifold topology, comprising six of the eight Thurston geometries.  These three-manifolds, defined as circle bundles over a two-dimensional orbifold, provide a valuable playground for the low-dimensional topologist, often being the choice for computations of a new three-manifold invariant.  
In this note, we focus on the case of Seifert fibered integral homology spheres.  Aside from $S^3$, a Seifert homology sphere $\Sigma$ is determined by $n$-relatively prime integers $a_1,\ldots,a_n$ with $n \geq 3$ and each $a_i \geq 2$.  In this paper, we will be interested in the homology cobordism type of $\Sigma(a_1,\ldots,a_n)$ and the concordance class of a Seifert fiber in $\Sigma(a_1,\ldots,a_n)$.  Even in this restricted class of three-manifolds, the theory is quite rich; for example, the homology spheres $\Sigma(2,3,6k-1)$ are all linearly independent in the homology cobordism group \cite{Fur}.

Our first theorem concerns the Heegaard Floer $d$-invariants of surgery on a Seifert fiber in a Seifert homology sphere.  Recall that $\Theta^H_3$ denotes the group of homology three-spheres modulo smooth homology cobordism, i.e. $Y_1$ and $Y_2$ are homology cobordant if they cobound a smooth homology $S^3 \times I$.  In \cite{OzSz1}, Ozsv\'ath and Szab\'o use Heegaard Floer homology to construct a surjective homomorphism $d:\Theta^H_3 \to 2\mathbb{Z}$ using Heegaard Floer homology.      

\begin{thm}\label{thm:d-fiber}
Let $Y = \Sigma(a_1,\ldots,a_n)$ be a Seifert homology sphere, oriented as the boundary of a negative definite plumbing.  If $K$ is a fiber in a Seifert fibration for $Y$, then $d(Y_{-1/m}(K)) = d(Y)$ for all integers $m \geq 0$.  Further, there are at most two values in the set $\{d(Y_{1/m}(K))\}_{m\in \mathbb{Z}}$.   
\end{thm}

\begin{rmk}
Note that the cardinality of $\{d(Y_{1/m}(K))\}_{m\in \mathbb{Z}}$ can be two, such as for the singular fiber of order 5 in $\Sigma(2,3,5)$, or one, as in the case of the singular fiber of order 7 in $\Sigma(2,3,5,7)$.  
\end{rmk}


While the above theorem does not give explicit values of the $d$-invariants of surgery or even how the two values might differ, this is still strong enough to provide certain four-dimensional obstructions.  Suppose that $Y_1$ and $Y_2$ are homology cobordant homology spheres.  Then two knots $K_i \subset Y_i$ are said to be {\em homology concordant} if they cobound a smooth annulus in some homology cobordism between $Y_1$ and $Y_2$. Note that for knots in $S^3$, homology concordance is a potentially weaker relation than the usual notion of smooth concordance, where such an annulus must sit in $S^3 \times I$.  

\begin{cor}\label{cor:d-concordance}
Let $J_m$ be the $(2,4m-1)$-cable of the negative torus knot $T_{-2,6m+1}$, where $m\geq 1$.  Then $J_m$ is not homology concordant to a Seifert fiber in any Seifert homology sphere.  More generally, for any homology sphere $Z$, if $Y = S^3_{1/n}(J_m) \# Z$ and $\tilde{J}$ is the connected sum of the core of surgery with the unknot in $Z$, then there is no homology concordance from $\tilde{J}$ to a Seifert fiber in any Seifert homology sphere.  
\end{cor}
\begin{proof}
Recall that if $L \subset M$ and $L' \subset M'$ are homology concordant, then $M_{1/p}(L)$ and $M'_{1/p}(L')$ are homology cobordant for all $p$.  Therefore, $\{d(M_{1/p}(L))\}_{p \in \mathbb{Z}} = \{d(M'_{1/p}(L'))\}\}_{p \in \mathbb{Z}}$ for all $p$.  It follows from \cite{wu} that $\{d(S^3_{1/p}(J_m))\}_{p \in \mathbb{Z}}$ consists of exactly three distinct values, and thus the same for $\{d(Y_{1/p}(\tilde{J}))\}_{p \in \mathbb{Z}}$, since $d(M \# M') = d(M) + d(M')$.  Theorem~\ref{thm:d-fiber} then completes the proof.
\end{proof}

The above examples were suggested to us by Marco Golla.

\begin{rmk}
Of course, there are several infinite families of non-trivial Seifert homology spheres that are homology cobordant to $S^3$ \cite{cassonharer}, so the above obstruction is not vacuous.  More artificially, for any Seifert homology sphere $\Sigma$, one can take $Z = \Sigma \# -S^3_{1/n}(J_m)$, so that $Y$ is necessarily homology cobordant to $\Sigma$.  
\end{rmk}

\begin{rmk}
All but one surgery on a fiber in a Seifert homology sphere will result in a Seifert manifold \cite{Heil}.  It is interesting to note that infinitely many surgeries on $J_m$ will be Seifert as well.  
\end{rmk}

Note that Corollary~\ref{cor:d-concordance} obstructs the singularity types of a PL disk that a Seifert fiber can bound in an acyclic four-manifold.  Recall that for any $p,q$ relatively prime, there is a Seifert fibration of $S^3$ such that the regular fibers are the $(p,q)$-torus knot $T_{p,q}$, and all such knots bound a PL disk in $B^4$.  However, only the unknot can bound a smooth disk in $B^4$, or even a topological locally flat one.  Likely known to experts, we extend this result to all Seifert fibers.   

\begin{thm}\label{thm:fiber-slice}
Let $Y$ be a Seifert fibered homology sphere and $K$ a Seifert fiber.  If $K$ is topologically slice in a topological homology ball, then $K$ is the unknot in $S^3$.  
\end{thm}
Note that by \cite{Freed1}, every homology three-sphere bounds a contractible (hence acyclic) topological four-manifold.  The proof of the above theorem will be completely elementary, using an appropriate generalization of the classical Fox-Milnor theorem.  

In light of Theorem~\ref{thm:d-fiber} and Theorem~\ref{thm:fiber-slice}, it's natural to ask if some surgery along a Seifert fiber is non-trivial in $\Theta^H_3$.  This is indeed true.  
\begin{prop}
Let $Y$ be a Seifert fibered homology sphere and $K$ a Seifert fiber.  For some $n$, $Y_{1/n}(K)$ is infinite order in the homology cobordism group, unless $K$ is the unknot in $S^3$.
\label{prop:hc}
\end{prop}
The above therefore provides an alternate proof that Seifert fibers (other than the unknot in $S^3$) are not {\em smoothly} slice in any homology ball.   Indeed, if $K$ were slice in some homology ball, then $Y_{1/n}(K)$ would bound a homology ball for all $n$. 

\begin{rmk}
In many instances $+1$-surgery on a Seifert fiber will change the Rokhlin invariant, and thus either $Y$ or $Y_1(K)$ will be infinite-order in the homology cobordism group, since the Neumann-Siebenmann invariant $\bar{\mu}$ will be non-zero \cite{savff}.  However, this is not the case for surgery along all Seifert fibers, such as the order 7 fiber in $\Sigma(3,5,7)$.  
\end{rmk}

\section*{Organization} In Section \ref{sec:sf}, we review Seifert fibered integral homology spheres and set notation and conventions.  The lemmas in Section \ref{sec:sf-surgery} describe how performing surgery on a fiber affects the Seifert invariants and in turn the plumbing graph.  There we also prove Proposition~\ref{prop:hc}.  In Section \ref{sec:d} we prove Theorem \ref{thm:d-fiber}.  Finally, we recall the Alexander polynomials of Seifert fibers from \cite{eisneum} and prove Theorem \ref{thm:fiber-slice} in Section \ref{sec:ap}.  

\section*{Acknowledgements} We thank Margaret Doig, Stefan Friedl, Marco Golla, Josh Greene, Andr\'as N\'emethi, Mark Powell, Danny Ruberman, and Nikolai Saveliev for helpful comments.  Additionally, we thank Chuck Livingston for his helpful feedback on an earlier draft of this paper.  We want to draw special attention to Cagri Karakurt who played an integral role in starting this project, and was part of the proof of Theorem~\ref{thm:d-fiber}.

%% file: background.tex
\section{Seifert fibered integral homology spheres}\label{sec:sf}
We recall some basic notions about Seifert fibered integral homology spheres and establish notation.  Consider $n$ pairs of relatively prime integers $(a_k, b_k)$ with $n \geq 3$, and an additional integer $e$.  Recall that the Seifert fibered space with base orbifold $S^2$ and Seifert invariants $e, (a_1,b_1), \ldots, (a_n,b_n)$ is the closed 3-manifold $Y = \Sigma(a_1, \ldots, a_n)$ constructed by starting with an $S^1$-bundle over an $n$-punctured $S^2$ of Euler number $e$ and filling the $k$th boundary component with an $a_k/b_k$-framed solid torus.  We refer to $e$ as the \emph{central Seifert invariant.}  Note that $Y$ inherits an $S^1$-action, the orbits of which are called fibers.  

The resulting manifold is an integral homology sphere if and only if
$$ a_1 \ldots a_n \left( \sum_{j = 1}^n \frac{b_j}{a_j} - e \right) = \pm 1.$$
In the present paper, we focus entirely on this case, and thus do not discuss other base orbifolds.  Choosing a sign for the right hand side amounts to fixing an orientation on $Y$, and we will always work with the orientation corresponding to $-1$.  Therefore, 
\begin{equation}
a_1 \ldots a_n \sum_{j = 1}^n \frac{b_j}{a_j} = -1 + a_1 \ldots a_n e.
\label{eq:inv}
\end{equation}

For each $j$, reducing the above equation modulo $a_j$ yields
\begin{equation}
\frac{a_1 \cdot \ldots \cdot a_n b_j}{a_j} \equiv -1 \pmod{a_j}.
\label{eq:mod}
\end{equation}

These equivalences imply that the integers $a_1, \ldots, a_n$ are pairwise relatively prime, and they furthermore completely determine the residue of $b_j$ modulo $a_j$ for each $j$.

Here we will take the convention that $1 \leq b_j < a_j$ when $a_j > 1$ and $b_j = 0$ when $a_j = 1$.  For each $k$ with $a_k \geq 2$, we refer to the core circle of the $\nicefrac{a_k}{b_k}$ Dehn filling as a \emph{singular fiber}, and refer to all other fibers as \emph{regular fibers}.  Lemma \ref{lem:cross} and Proposition \ref{prop:inv} below will describe how surgery on a fiber affects the Seifert invariants.  For notational convenience, we always choose the fiber of order $a_n$.  This fiber is regular if and only if $a_n =1$, and we will require that $a_k > 1$ for $k < n$. It is important to notice that if a Seifert fibered homology sphere $Y$ has fewer than three singular fibers, then $\Sigma \cong S^3$.  Furthermore, $\Sigma(a_1, \ldots, a_n, 1) \cong \Sigma(a_1, \ldots, a_n)$.  Finally, since any fiber in $\Sigma(1) = S^3$ is unknotted, we are easily able to omit the case of $n = 1$ for the rest of the paper.

Given a Seifert fibered integral homology sphere $Y =\Sigma(a_1, \ldots, a_n)$ with Seifert invariants $e, (a_1,b_1), \ldots, (a_n, b_n)$, consider for each $k$ with $a_k \geq 2$, the continued fraction expansion
$$ \frac{a_k}{b_k} = [x_{k,1},  \ldots, x_{k, m_k}] := x_{k,1} - \frac{1}{ x_{k,2} - \frac{1}{ \ldots - \frac{1}{x_{k, m_k} } } },$$
where each $x_{k,i}$ is at least $2$.

The \emph{plumbing graph} for $Y$ is the star-shaped weighted graph $\Gamma(Y)$ consisting of $n$ chains of vertices emanating from a central vertex with weight $-e$.  The $k$th chain consists of $m_k$ vertices carrying weights $-x_{k,1}, \ldots, -x_{k,m_k}$, numbered from the center outward. If $a_n = 1$, there is no $n$th chain.  This graph induces a negative-definite plumbed four-manifold $P(\Gamma(Y))$.  By replacing the weighted vertices in the graph with framed unknots (linked exactly when they share an edge), one obtains a Kirby diagram for $P(\Gamma(Y))$ and a Dehn surgery diagram for $\partial P(\Gamma(Y)) \cong Y$.  See Figure \ref{fig:dehn}.

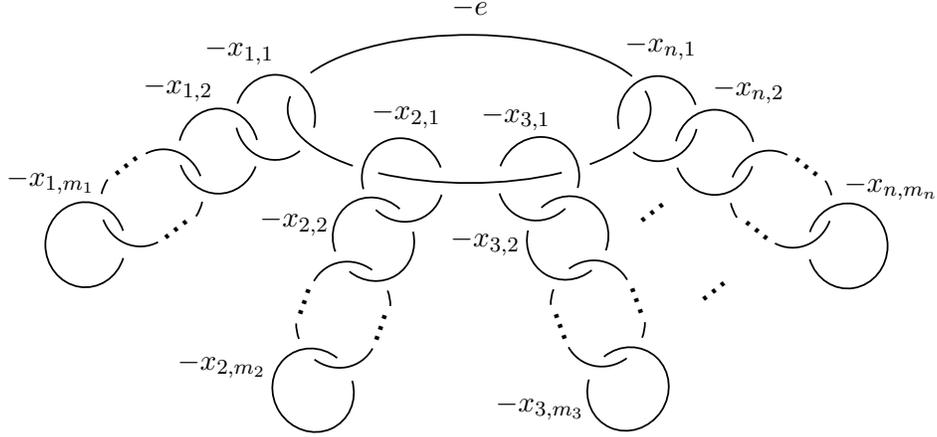
\begin{figure}
\labellist
\pinlabel $-e$ at 213 207
\pinlabel $-x_{1,1}$ at 108 187
\pinlabel $-x_{1,2}$ at 80 170
\pinlabel $-x_{1,m_1}$ at 22 128
\pinlabel $-x_{2,1}$ at 184 158
\pinlabel $-x_{2,2}$ at 133 109
\pinlabel $-x_{2,m_2}$ at 100 44
\pinlabel $-x_{3,1}$ at 234 157
\pinlabel $-x_{3,2}$ at 220 100
\pinlabel $-x_{3,m_3}$ at 245 25
\pinlabel $-x_{n,1}$ at 300 189
\pinlabel $-x_{n,2}$ at 340 169
\pinlabel $-x_{n,m_n}$ at 405 125
\endlabellist
\includegraphics[width = 0.75\linewidth]{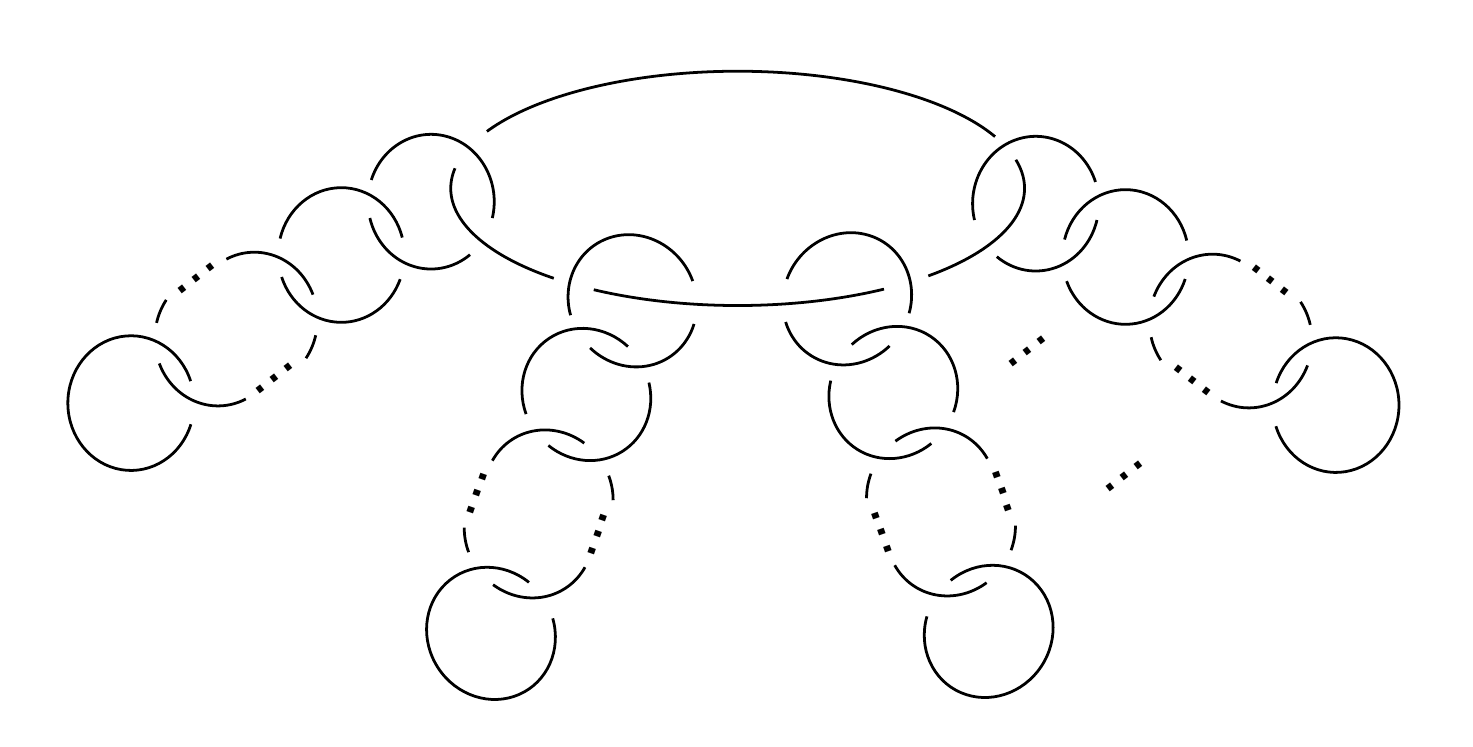}
\label{fig:dehn}
\caption{A Dehn surgery diagram for the Seifert homology sphere $\Sigma(a_1,\ldots,a_n)$ with Seifert invariants $e, (a_1,b_1),\ldots,(a_n,b_n)$ where $\nicefrac{a_k}{b_k} = [x_{k,1},\ldots,x_{k,m_k}]$.}
\label{fig:dehn}
\end{figure}

%% file: surgery-new.tex
\section{The effects of surgeries on Seifert invariants}\label{sec:sf-surgery}

Given a Seifert-fibered homology sphere $\Sigma(a_1, \ldots, a_n)$, we remind the reader that we require $a_k \geq 2$ for $1 \leq k \leq n-1$.  We allow $a_n = 1$, in which case require $b_n = 0$ and the fiber associated to $a_n$ is in a regular fiber.  We establish the following notation, which will be used throughout the rest of the paper:
$$ \alpha:= \prod_{k=1}^{n-1} a_k \quad \text{and} \quad \beta:= \frac{\alpha b_n + 1}{a_n}.$$ 
Note that due to \eqref{eq:mod}, $\beta$ is an integer.
In this section, we prove Proposition~\ref{prop:hc}, by studying the behavior of Seifert invariants under surgery along fibers.  The proof of Proposition \ref{prop:hc} will rely on the following:
\begin{thm}[\cite{NeumannZagier}]
Let $Y$ be a Seifert-fibered homology sphere with Seifert invariants $e$, $(a_1, b_1), \ldots, (a_n, b_n)$.   If $e > 1$, then the class of $Y$ has infinite order in the integral homology cobordism group.
\label{thm:hc}
\end{thm}

We will prove Proposition~\ref{prop:hc} by showing that there exists a surgery on a fiber which results in a Seifert homology sphere with $e > 1$.  Therefore, we must describe the result of performing $\nicefrac{1}{m}$-surgery on a fiber in a Seifert fibered integral homology sphere, which will also be  used in the proof of Theorem~\ref{thm:d-fiber}.  This is understood according to the following well-known formula:
\begin{lem}
Let $Y= \Sigma \left( a_1, \ldots, a_n\right)$ be a Seifert fibered integral homology sphere with $n \geq 3$.  Choose $K \subset Y$ to be the fiber of order $a_n$, and fix an integer $m$.  Then
$$Y_{1/m}(K) \cong \sigma \Sigma \left( a_1, \ldots, a_{n-1}, \abs{a_n - m \alpha}\right),$$
where $\sigma$ denotes the sign of the nonzero integer $a_n  - m \alpha$.  The core of the surgery is the fiber of order $a_n - m \alpha$.  
\label{lem:seif}
\end{lem}
%

In particular, if $\Sigma(a_1, \ldots, a_n)$ has the property that $a_n < \alpha$, then the result of $+1$-surgery on the fiber associated to $a_n$ is $-\Sigma(a_1, \ldots, a_{n-1}, \alpha - a_n)$.  Note that this results in a Seifert fibered integral homology sphere with the opposite of our usual orientation.

\begin{lem}\label{lem:cross}
Let $\Sigma(a_1, \ldots, a_n)$ be a Seifert fibered integral homology sphere with Seifert invariants $e,  (a_1,b_1), \ldots, (a_n, b_n)$.

If $a_n < \alpha$, then the Seifert fibered integral homology sphere $\Sigma(a_1, \ldots, a_{n-1}, \alpha-a_n)$ has Seifert invariants given by
$$n-e,  (a_1,a_1-b_1), \ldots, (a_{n-1},a_{n-1}-b_{n-1}), (\alpha - a_n, b_n - \beta + \alpha - a_n).$$
\end{lem}
\begin{proof}
Note that $\beta$ is an integer by \eqref{eq:mod}.  We first claim that $b_n' := b_n - \beta + \alpha - a_n$ satisfies the conventions described in Section \ref{sec:sf}.
If $\alpha - a_n > 1$, then
$$\beta = \frac{\alpha b_n + 1}{a_n} > \frac{a_n b_n+1}{a_n} > b_n \quad \text{and so} \quad b_n' < \alpha - a_n.$$
Furthermore,
\begin{align*}
b_n' = b_n - \beta + \alpha - a_n &= \frac{a_n b_n - \beta a_n + \alpha a_n - a_n^2}{a_n}
=\frac{a_n b_n - \alpha b_n - 1 + \alpha a_n - a_n^2}{a_n}\\
&=\frac{(a_n - b_n)(\alpha - a_n) - 1}{a_n} \geq \frac{2(a_n - b_n) - 1}{a_n} > 0.
\end{align*}
Therefore, $0 < b_n' < \alpha-a_n$ in this case.  On the other hand, if $\alpha - a_n = 1$, then we claim that $b_n' = 0$.  First notice that
$$ \alpha (a_n - 1) \equiv -\alpha \equiv -1 \pmod{a_n} \quad \text{and thus} \quad b_n = a_n - 1.$$
Indeed, it follows that
$$\beta = \frac{\alpha b_n + 1}{a_n} = b_n + \frac{b_n+1}{a_n} = b_n + 1 \quad \text{and so} \quad \ b_n' = b_n - \beta + \alpha - a_n = -1+1 = 0.$$

Now it suffices to prove that 
\begin{equation}\label{eq:crosszero}
\alpha(\alpha - a_n) \left( \sum^{n-1}_{j=1} \frac{a_j - b_j}{a_j} + \frac{b_n - \beta + \alpha - a_n}{\alpha-a_n} - (n-e) \right) = -1.  
\end{equation}
First, by reversing orientation on $\Sigma(a_1,\ldots,\alpha - a_n)$, observe that \eqref{eq:crosszero} holds if and only if 
\begin{equation*}\label{eq:crosszero-neg}
\alpha(\alpha - a_n) \left( \sum^{n-1}_{j=1} \frac{b_j}{a_j} + \frac{\beta - b_n}{\alpha-a_n} - e \right) = +1.
\end{equation*}

Therefore, we will show that
\begin{equation}\label{eq:crosszero-combine}
\alpha a_n \left(\sum^n_{j=1} \frac{b_j}{a_j} - e\right) + \alpha(\alpha - a_n) \left( \sum^{n-1}_{j=1} \frac{b_j}{a_j} + \frac{\beta - b_n}{\alpha - a_n} -e \right) = 0,
\end{equation}
which will establish the result.  Rearranging the left-hand side of \eqref{eq:crosszero-combine}, we obtain
\begin{equation}\label{eq:crosszero-simplified}
\alpha a_n \left(\frac{b_n}{a_n} - \frac{\beta - b_n}{\alpha - a_n}\right ) + \alpha^2\left (\sum^{n-1}_{j=1} \frac{b_j}{a_j} - e\right) + \alpha^2 \frac{\beta - b_n}{\alpha - a_n}.
\end{equation}
Note that the first term is precisely $-\frac{\alpha}{\alpha - a_n}$, while the second term is $-\alpha \beta$.  For the latter claim, simply use that
$$
\alpha\left (\sum^{n-1}_{j=1} \frac{b_j}{a_j} - e \right) = -\frac{\alpha b_n}{a_n} - \frac{1}{a_n}.
$$  
It is now straightforward to show that the sum in \eqref{eq:crosszero-simplified} is equal to zero.  
\end{proof}

\begin{proof}[Proof of Proposition \ref{prop:hc}]

Note that the conclusion of the proposition is independent of the orientation, so we choose $Y$ to have the orientation as given in Section~\ref{sec:sf}.  We first address the case in which $Y= S^3$ and $K$ is a fiber other than the unknot.  In this case, $Y_{+1}(K) = -\Sigma(p,q,pq-1)$ for some $p, q \geq 2$.  Each of these manifolds is infinite order in the homology cobordism group \cite{Fur}.  

Therefore, we now assume that there are at least three singular fibers of order at least two and let $K$ denote the fiber of order $a_n$.  Let $a_n = q\alpha + r$ for integers $q,r$ with $q \geq 0$ and $0 < r < \alpha$.  Then Lemma \ref{lem:seif} implies that
$$ Y_{1/(q+1)}(K) \cong - \Sigma(a_1, \ldots, a_{n-1}, \alpha - r).$$
According to Lemma \ref{lem:cross}, this manifold has central Seifert invariant equal to $n - e$.  The proposition follows now from Theorem \ref{thm:hc} above, as at least one of $e$ and $n-e$ is greater than 1.
\end{proof}

In the above, we did not need to understand the behavior of the Seifert invariants for all $1/m$-surgeries on a fiber.  However, for completeness, we give the general description of this below.  The proof is similar to that of Lemma~\ref{lem:cross}.    

\begin{prop}
Let $Y= \Sigma(a_1,\ldots,a_n)$ be a Seifert fibered integral homology sphere with Seifert invariants $e, (a_1,b_1), \ldots, (a_n,b_n)$, $n \geq 3$.  Let $K$ denote the fiber of order $a_n$.  Then, the following hold for any $m \in \Z$:
\begin{enumerate}
\item If $a_n \geq  m \alpha$, then $Y_{1/m}(K)$ has Seifert invariants given by 
$$
e, (a_1,b_1), \ldots, (a_{n-1}, b_{n-1}), (a_n - m\alpha, b_n - m\beta).  
$$
\item If $a_n < m \alpha$, then $-Y_{1/m}(K)$ has Seifert invariants given by 
$$
n-e, (a_1,a_1-b_1),\ldots, (a_{n-1}, a_{n-1}-b_{n-1}), (m\alpha - a_n, m(\alpha - \beta) + b_n - a_n).  
$$
\end{enumerate}
\label{prop:inv}
\end{prop}

\begin{ex}
Let $q > p \geq 2$ be a pair of relatively prime integers, and fix integers $n \geq 3$ and $m > 0$.  Let $Y$ be any Seifert fibered homology sphere which can be obtained from $\Sigma(p,q,pqm-1)$ by a sequence of modifications of the form
$$ \Sigma(a_1, \ldots, a_n) \rightarrow \Sigma\left( a_1, \ldots, a_n + \alpha \right) \quad \text{or} \quad \Sigma(a_1, \ldots, a_{n-1}) \rightarrow \Sigma\left( a_1, \ldots, a_{n-1}, \alpha + 1\right).$$
By Proposition \ref{prop:inv} the manifold $Y$ has the same central Seifert invariant as $\Sigma(p,q,pq-1)$, which is equal to $2$.  Therefore, $Y$ has infinite order in the homology cobordism group by Theorem \ref{thm:hc}.
\label{ex:inford}
\end{ex}

%% file: dinvariants.tex
\section{The $d$-invariant of surgery on a Seifert fiber}\label{sec:d}

The $d$-invariant of a Seifert fibered integral homology sphere can be calculated using a unimodular negative-definite integral lattice induced by its plumbing graph \cite{os:plumb}.  We now recall this process.  

Given a plumbing graph $\Gamma$ with vertices $v_1, \ldots, v_n$, we can associate to $\Gamma$ a lattice $L(\Gamma)$ as follows:
$$ L(\Gamma) = \text{span}_{\mathbb{Z}}(v_1, \ldots, v_n) \text{ and } \pair{v_i, v_j}_{L(\Gamma)} = \begin{cases}
e_i & \text{ if $i=j$}\\
1 & \text{ if $i\neq j$ and $v_i$ and $v_j$ share an edge in $\Gamma$}\\
0 & \text{otherwise}.
\end{cases}
$$
Letting $P(\Gamma)$ denote the plumbed four-manifold associated to $\Gamma$, it is easy to see that $L(\Gamma)$ is isomorphic to the lattice consisting of $H_2(P(\Gamma); \mathbb{Z})$ equipped with its intersection form.  Notice that this lattice is unimodular if and only if the 3-manifold $\partial P(\Gamma)$ is an integral homology sphere.

Given a unimodular integral lattice $L$, the set of \emph{characteristic vectors} of $L$ is the set
$$ \text{Char}(L) = \left \{ \chi \in L :\pair{\chi, y}_L \equiv |y|_L \pmod 2 \text{ for all } y \in L\right\}$$
which is clearly a coset in $L/2L$.

Given a negative-definite unimodular integral lattice $L$, we define its \emph{lattice $d$-invariant} by
$$ d(L) = \text{max} \left\{ \frac{\abs{\chi}_L + \text{rank}(L)}{4} : \chi \in \text{Char}(L) \right\} \in 2\mathbb{Z}.$$

We recall and collect several relevant facts about the lattice $d$-invariant:
\begin{enumerate}[(i)]
\item If $-\mathbb{Z}^n$ denotes the negative-definite unimodular diagonalizable rank-$n$ lattice, $d(-\mathbb{Z}^n) = 0$.
\item Given two unimodular integral lattices $L$ and $L'$, $d(L \oplus L') = d(L) + d(L')$.
\item Given a Seifert fibered homology sphere $Y$, oriented as in Section~\ref{sec:sf}, the Heegaard-Floer $d$-invariant of $Y$ coincides with the lattice-theoretic d-invariant arising from the plumbing graph of $Y$, i.e. $d(L(\Gamma(Y))) = d(Y)$.
\end{enumerate}

The first two facts above follow easily from the definitions, and the third is due to \cite{os:plumb}.  With this, we are able to study the $d$-invariants of Seifert homology spheres as we vary the Seifert invariants.  

\begin{prop}
If $Y = \Sigma(a_1,\ldots,a_{n-1}, a_n)$ and $Y' = \Sigma(a_1,\ldots,a_{n-1},a'_n)$ are Seifert homology spheres with $a_n \equiv a_n' \pmod{\alpha}$, then $d(Y) = d(Y')$.\label{prop:d1}
\end{prop}

\begin{ex}
The manifold $\Sigma(5,7,11)$ has d-invariant equal to $2$, and thus has infinite order in the integral homology cobordism group.  However, this fact is not detected by its central Seifert invariant, which is indeed equal to $1$.  By Theorem \ref{thm:hc} and Proposition \ref{prop:d1}, the same can be said for any Seifert fibered homology sphere obtained from $\Sigma(5,7,11)$ via a sequence of the modifications for Seifert invariants described in Example \ref{ex:inford} above.
\end{ex}

\begin{rmk}
It is natural to ask whether a Seifert homology sphere with central Seifert invariant different than $1$ must have non-vanishing d-invariant.  Lecuona and Lisca \cite[Lemma 3.3]{LecuonaLisca} show that if a Seifert homology sphere has $n$ singular fibers and central Seifert invariant equal to $n-1$ (the largest possible value for such a manifold), then its d-invariant is non-vanishing.  Therefore, a Brieskorn sphere ($n=3$) with central Seifert invariant different than $1$ has non-vanishing d-invariant.
\end{rmk}

%

\begin{proof}[Proof of Proposition \ref{prop:d1}]
It suffices to prove that $d(Y) = d(Y')$, where $Y = \Sigma(a_1, \ldots, a_{n-1}, a_{n})$ and $Y' = \Sigma(a_1, \ldots, a_{n-1}, a_n+\alpha)$.   Note that $Y'$ is obtained from $Y$ by $-1$-surgery on the singular fiber of order $a_n$ by Lemma~\ref{lem:seif}.  Let $K$ denote this fiber.

If $n \leq 2$, then we see that $Y = Y' = S^3$ and the claim holds.  Henceforth we restrict our attention to the case of $n \geq 3$.  
 

We begin with the case in which $a_n > 1$.  Let $\Gamma$ denote the associated plumbing graph of $Y$.  In this case, the fiber of order $a_n$ is well-known to be represented by a meridian of the last component on the $n$th arm in the surgery diagram in Figure \ref{fig:dehn}, and thus surgery corresponds to appending a vertex $v$ to the end of the $n$th arm of $\Gamma$.  Since $-1$-surgery corresponds to attaching a 2-handle, the weight of this new vertex $v$ must indeed be integral, and we obtain a new plumbing graph, $\Gamma'$.   However, to use this plumbing graph to compute $d(Y')$, we must show that $P(\Gamma')$ is negative definite and that the weight of $v$ is at most $-2$.

Because we are attaching a $-1$-framed 2-handle to a homology sphere, we have that the intersection form of $P(\Gamma')$ splits over the 2-handle addition.  Therefore, $P(\Gamma')$ has intersection form $L(\Gamma) \oplus -\mathbb{Z}$, and consequently is negative definite.  It follows that the weight of $v$ is at most -1.

We shall now show that the weight of $v$ cannot be equal to $-1$. Recall that the fiber slope on the exterior of $K$ is the slope on the boundary induced by the Seifert fibration on $Y$.   In the standard meridian-longitude coordinates on $K$ in the surgery diagram in Figure \ref{fig:dehn}, the fiber slope $\phi$ is exactly $-\nicefrac{b^*_n}{a_n}$, where $b^*_n$ is the inverse of $b_n$ mod $a_n$.  To see this, note that surgery corresponding to the fiber slope always results in a connected sum of lens spaces and/or $S^2 \times S^1$'s \cite{Heil}.  Except for the twisted I-bundle over the Klein bottle, no Seifert manifold with torus boundary admits more than one reducible filling.  Notice that for homology reasons, this manifold cannot be the exterior of a knot in a homology sphere.  Finally, it can be seen through Kirby calculus in Figure~\ref{fig:dehn} to show that $-\nicefrac{b^*_n}{a_n}$-surgery on $K$ results in a connected sum of lens spaces.  Therefore, this must be the fiber slope.  

Suppose that the weight of $v$ is equal to $-1$.  It follows that the distance from $\phi$ to the slope $-\nicefrac{1}{1}$ is $|b_n - a_n|$, which is at most $a_n$.  This contradicts the fact that the distance specifies the order of the new singular fiber obtained from surgery on $K$, namely $a_n + \alpha > a_n$.

Indeed, we may now compute $d(Y')$ using $\Gamma'$ and thus we have $$d(Y') = d(L(\Gamma')) = d(L(\Gamma)) + d(-\mathbb{Z}) = d(L(\Gamma)) = d(Y).$$  

The case in which $a_n = 1$ involves instead appending a new vertex to the central vertex, and the proof is analogous to the previous case.  In this case, the fiber slope is 0.  If $a_n = 1$, then $b_n = 0$ by our conventions, and so $|b_n - a_n| = a_n$.
\end{proof}

We are now prepared to state the following more detailed version of Theorem \ref{thm:d-fiber} from the introduction.
\begin{prop}\label{prop:d-twovalues}
Let $Y= \Sigma(a_1, \ldots, a_n)$ be a Seifert fibered integral homology sphere, where we assume that $a_j > 1$ for $1 \leq j \leq n-1$.  Let $K$ denote the fiber of order $a_n$ and let $r > 0$ denote the residue of $a_n$ modulo $\alpha$.  For any integer $m$ we have that
$$d\left( Y_{1/m}(K) \right) = \begin{cases}
d(Y) & \text{if} \quad a_n -m\alpha > 0\\
- d(\Sigma(a_1, \ldots, a_{n-1}, \alpha - r)& \text{otherwise}.
\end{cases}$$

In particular, for any Seifert fiber $K$, the set
$$ \left\{ d(Y_{1/m}(K)) :  m \in \mathbb{Z} \right\} \subset \mathbb{Z}$$
contains at most two elements.
\label{prop:d2}
\end{prop}

\begin{proof}
If $a_n - m\alpha > 0$, Lemma \ref{lem:seif} implies that
$$ Y_{+1/m}(K) \cong \Sigma(a_1, \ldots, a_{n-1}, a_n - m\alpha).$$
Since $a_n - m\alpha \equiv a_n \pmod{\alpha}$, $d(Y) = d(Y_{-1/m}(K))$ by Proposition \ref{prop:d1}.

Note that it is not possible to have $a_n - m\alpha = 0$ since $a_n$ and $\alpha$ are relatively prime and $a_j > 1$ for $j < n$.  If $a_n - m\alpha < 0$, Lemma \ref{lem:seif} implies that
$$ Y_{+1/m}(K) \cong -\Sigma(a_1, \ldots, a_{n-1}, \abs{a_n - m\alpha}) = -\Sigma(a_1, \ldots, a_{n-1}, -(a_n-m\alpha)).$$
Since $ -(a_n + m\alpha) \equiv \alpha - r \pmod{\alpha}$, the result again follows from Proposition \ref{prop:d1}.
\end{proof}

\begin{rmk}
Neumann \cite{neum} and Siebenmann \cite{sieb} independently defined an invariant $\overline{\mu} \in \mathbb{Z}$ for graphmanifold homology spheres.  For any Seifert homology sphere with orientations as in this paper, we have that $d(Y) \geq -2\overline{\mu}(Y)$.

Let $Y= \Sigma(2,3,5)$ and let $K \subset Y$ be the singular fiber of order $5$.  Then we have that for every positive integer $n$, $ Y_{1/n}(K) \cong -\Sigma(2,3,6n-5)$.
Therefore, $d(Y) = 2$ while $d(Y_{1/n}(K)) = 0$ for every $n\geq 1$.  It is interesting to contrast this with the behavior of the invariant $\overline{\mu}$ with respect to the same surgeries.  It is easy to verify that $\overline{\mu}(Y) = -1$.  By applying Theorem 5.1 of \cite{neum} and the fact that $\Sigma(2,3,1) \cong S^3$, we can see that for $n \geq 1$,
$$ \overline{\mu}(Y_{1/n}(K)) = - \overline{\mu}(\Sigma(2,3,6n-5)) = \begin{cases}
1 & \text{if $n$ is even}\\
0 & \text{if $n$ is odd}. 
\end{cases}$$
\end{rmk}

Finally, it is natural to ask if an analogue of the results here applies to the refined $\overline{d}$-, $\underline{d}$-invariants of Hendricks-Manolescu \cite{HendricksManolescu}.  By the work of Dai-Manolescu \cite[Theorem 1.2]{DaiManolescu}, for Seifert homology spheres with the orientation conventions given here, $\underline{d}(Y) = -2\overline{\mu}(Y)$ and $\overline{d}(Y) = d(Y)$, and thus $\underline{d}$ can change under negative surgery on a Seifert fiber.   

%% file: topslice.tex
\section{Nontrivial Seifert fibers are not topologically slice}\label{sec:ap}

We shall now develop a proof of Theorem \ref{thm:fiber-slice}, that non-trivial Seifert fibers are not topologically slice, which will use the Fox-Milnor sliceness obstruction for the symmetrized Alexander polynomial.

Given a knot $K$ in an integral homology sphere $Y$, we let $\widetilde{\Delta}_{K}(t)$ denote the symmetrized Alexander polynomial of $K$ (a Laurent polynomial with symmetric coefficients), and we let $\Delta_{K}(t)$ denote the unsymmetrized Alexander polynomial, that is
$$\Delta_{K}(t) = t^m \cdot \widetilde{\Delta}_{K}(t) \quad \text{where $m$ is the degree of $\widetilde{\Delta}_{K}$}.$$

Using Milnor's duality formula for torsions \cite{Mi4}, Fox and Milnor proved that if a knot $K \subset S^3$ is topologically slice in the four-ball, then the symmetrized Alexander polynomial of $K$ is of the form $g(t)g(t^{-1})$ for some polynomial $g \in \mathbb{Z}[t]$ \cite{foxmilnor}.  Milnor's duality formula and the Fox-Milnor theorem have been generalized in several directions, see e.g. \cite{tur}, \cite{kl}.  The appropriate generalization for our purposes will be proven in a forthcoming paper of Friedl, Kim, Nagel, Orson, and Powell \cite{FKNOP}: if a knot $K$ in an integral homology sphere $Y$ bounds a properly and locally flatly embedded disk in a topological homology ball bounded by $Y$, then its Alexander polynomial satisfies the Fox-Milnor condition.
%
%

Theorem~\ref{thm:fiber-slice} is well-known for torus knots, so we will only focus on Seifert-fibered homology spheres $\Sigma(a_1, \ldots, a_n)$ with at least three singular fibers.  Throughout this section, we use $S$ to denote the numerical semigroup generated by the integers $\frac{\alpha}{a_1}, \ldots, \frac{\alpha}{a_{n-1}}$.

Eisenbud and Neumann \cite{eisneum} compute the Alexander polynomial of the fiber of order $a_{n}$ in $\Sigma(a_1, \ldots, a_{n})$:
\begin{equation}
\Delta_K(t) = \frac{(1-t^{\alpha})^{n-2}(1-t)}{\prod\limits_{j=1}^{n-1} \left(1-t^{\frac{\alpha}{a_j}}\right)},
\label{poly}
\end{equation}
which we note is independent of the value of $a_{n}$.  In the case of $n = 3$, this recovers the familiar formula for the Alexander polynomial for torus knots in $S^3$.

\begin{prop}
The non-zero coefficients of $\Delta_K(t)$ are all equal to $\pm 1$.  \label{prop:polyco}
\end{prop}

Prior to proving the proposition, we develop an important lemma.  Observe that 
\[
\frac{1}{\prod\limits_{j=1}^{n-1} \left(1-t^{\frac{\alpha}{a_j}}\right)} = \sum_{s \in S} m(s) t^s,
\]
where the coefficient $m(s)$ is defined to be the number of distinct ways in which $s$ can be expressed as a numerical semigroup element.

Notice that given $s \in S$, since $a_1,\ldots, a_{n-1}$ are relatively prime, we have $m(s) > 1$ if and only if $s = s' + \alpha$ for another element $s' \in S$.  In particular, for each positive integer $k$ the numerical semigroup element $s = k\alpha$ has many expressions:

\begin{lem} For $1 \leq k \leq n-1$, we have $\displaystyle m(k\alpha) = { k+n-2 \choose n-2}$.
\end{lem}
\begin{proof}
Given positive integers $k$ and $j$, let $P_{k,j}$ denote the set of ordered partitions of the number $k$ into $j$ non-negative integers:
$$ P_{k,j} := \left\{ (t_1, \ldots, t_j) \mid t_1 + \ldots + t_j = k, \ t_j \geq 0 \right\}.$$
The cardinality $|P_{k,j}|$ is equal to the number of ways to distribute $k$ identical items among $j$ distinguishable bins, which is well known to be equal to the binomial coefficient ${ k+j-1 \choose j-1}$.

Furthermore, for any numerical semigroup element $s \in S$, let $P(s)$ denote the set of ways to represent $s$:
$$ P(s) := \left\{ (x_1, \ldots, x_{n-1}) \mid  \alpha \left( \frac{x_1}{a_1} + \ldots + \frac{x_{n-1}}{a_{n-1}} \right) = s, \ x_j \geq 0 \right\}.$$
Of course, the cardinality of $P(s)$ is equal to $m(s)$.  We will construct a bijection from $P_{k,n-1}$ to $P(k\alpha)$, thus proving the claim.
Let $f:P_{k,n-1} \rightarrow P(k \alpha)$ be the function
$$f(t_1, \ldots, t_{n-1}) = (t_1a_1,\ldots,t_{n-1}a_{n-1}).$$
It is obvious that $f$ is one-to-one, so it remains to show that $f$ is onto.

Consider $(x_1, \ldots, x_{n-1}) \in P(k \alpha)$.  The equation 
$$\alpha \left( \frac{x_1}{a_1} + \ldots + \frac{x_{n-1}}{a_{n-1}} \right) = k \alpha$$
implies that for each $j$, $a_j $ divides $x_j$, as the $a_j$ are relatively prime.  It follows that $(x_1, \ldots, x_{n-1})$ is in the image of $f$.
\end{proof}

\begin{proof}[Proof of Proposition \ref{prop:polyco}]
Let $S_*$ be the subset of $S$ consisting of numerical semigroup elements which have a unique expression.  Note that this is not only the elements less than $\alpha$.  By writing any element of $S$ with a non-unique decomposition as $s = k\alpha + s'$ where $s' \in S_*$, we can rewrite 
\begin{align*}
\frac{1}{\prod\limits_{j=1}^{n-1} \left(1-t^{\frac{\alpha}{a_j}}\right)}  &= \sum_{s \in S_*} t^s + m(\alpha) \sum_{s \in S_*} t^{s + \alpha} + m(2\alpha) \sum_{s \in S_*} t^{s+2\alpha} + ... \\
&= \sum^\infty_{k=0} m(k\alpha) t^{k\alpha} \left(\sum_{s \in S_*} t^s \right) \\
&=  \sum^\infty_{k=0}  { k+n-2 \choose n-2} t^{k\alpha} \left(\sum_{s \in S_*} t^s \right)\\
&= \frac{1}{(1-t^{\alpha})^{n-1}} \left(\sum_{s \in S_*} t^s \right).
\end{align*}

Therefore, we get by \eqref{poly}:
\begin{align*}
\Delta_K(t) = \frac{1-t}{1-t^{\alpha}} \sum_{s \in S_*} t^s &= \frac{1}{1-t^{\alpha}} \sum_{s \in S_*} t^s - t^{s+1} \\
&= \sum^\infty_{k = 0} t^{k\alpha}  \left(\sum_{s \in S_*} t^s - t^{s+1}  \right).
\end{align*}
 We would like to see that all the non-zero coefficients are $\pm 1$.  Note that the non-zero coefficients of $\sum_{s \in S_*} t^s - t^{s+1}$ are all $\pm 1$.  Therefore, in order for there to be some non-zero coefficient of $\Delta_K(t)$ with value other than $\pm 1$, there would have to be two elements, $s$ and $s'$, of $S_*$ which differ by a multiple of $\alpha$.  Without loss of generality $s - s' = k\alpha$ for some $k \geq 1$.  Then, $s$ has a non-unique decomposition in $S$ and thus $s$ is not in $S_*$, a contradiction.  Therefore, the non-zero coefficients must be $\pm 1$.     
\end{proof}

\begin{proof}[Proof of Theorem \ref{thm:fiber-slice}]
Let $K$ be a fiber in a Seifert homology sphere, other than that the unknot in $S^3$.  We will show that $\widetilde{\Delta}_K$ fails the generalized Fox-Milnor condition.  Towards contradiction, suppose that such a factorization of $\widetilde{\Delta}_K$ exists with $f(t) = a_m t^m + \ldots + a_1 t + a_0$.  Notice that the constant coefficient of $f(t)f(t^{-1})$ is equal to $a_m^2 + \ldots + a_1^2 + a_0^2$.  It can be deduced from Equation \ref{poly} that the unsymmetrized polynomial $\Delta_K(t)$ is not a constant, implying that at least two of the coefficients $a_m, \ldots, a_1, a_0$ are nonzero.  Therefore the aforementioned quantity $a_m^2 + \ldots + a_1^2 + a_0^2$ is strictly greater than 1, contradicting Proposition \ref{prop:polyco}.
\end{proof}